%% file: main_jcam.tex
\documentclass[preprint]{elsarticle}
\usepackage{epsfig}
\usepackage{amsmath}
\usepackage{epic}
\usepackage{amssymb}
\usepackage{color}

\newcommand{\Hull}{\operatorname{Hull}}

\newcommand{\abs}[1]{\lvert#1\rvert}

\newcommand{\innprd}[2]{\left( #1 , #2 \right)}

\newcommand{\EE}{\mathbb E}

\def\norm#1{\left|\!\left| #1 \right|\!\right|}
\def\nnorm#1{|\!| #1 |\!|}

\def\op#1{{\mathcal #1}}

\def\Re{I\!\!R}

\def\be{\begin{eqnarray}}
\def\ee{\end{eqnarray}}
\def\bes{\begin{eqnarray*}}
\def\ees{\end{eqnarray*}}
\def\R{\mathbb{R}}
\def\C{\mathbb{C}}
\def\N{\mathbb{N}}

\def\abs#1{\lvert#1\rvert}
\newtheorem{theorem}{Theorem} 
\newtheorem{lemma}[theorem]{Lemma} 
\newtheorem{proposition}[theorem]{Proposition} 
\newtheorem{corollary}[theorem]{Corollary} 
\newdefinition{remark}{Remark} 
\newproof{proof}{Proof}
\newproof{pot}{Proof of Theorem \ref{thm2}}
\newtheorem{conjecture}[theorem]{Conjecture}

\journal{J. Comput. Appl. Math.}

\begin{document}

\begin{frontmatter}

\title{Sharp estimates for the convergence rate of Orthomin(k) for a  class of linear systems\tnoteref{t1}} 
\tnotetext[t1]{This work was supported the National Science Foundation under award number 2409951.} 


\author[1]{Andrei Dr{\u{a}}g{\u{a}}nescu\corref{cor1}} 
\ead{draga@umbc.edu}
\affiliation[1]{organization={Department of Mathematics and Statistics, University of Maryland, Baltimore
     County}, 
   addressline={1000~Hilltop Circle},
   postcode={MD, 21250}, city={Baltimore}, country={USA}}

\author[2]{Florin Spinu}
\ead{fspinu@gmail.com}
\affiliation[2]{
          city={New York},
          country={USA}}

\cortext[cor1]{Corresponding author}

\begin{abstract}
In this work we show that the convergence rate of Orthomin($k$)
applied to systems of the form $(I+\rho U) x = b$, where $U$ is a
unitary operator and $0<\rho<1$, is  less than or equal to
$\rho$. Moreover, we give examples of operators $U$ and $\rho>0$ for
which the asymptotic convergence rate  of Orthomin($k$) is exactly $\rho$, thus
showing that the estimate is sharp.   While the systems under scrutiny
may not be of great interest in themselves, their existence shows that, in general,
Orthomin($k$) does not converge faster than Orthomin(1). Furthermore, we give
examples of systems for which Orthomin($k$) has the same asymptotic convergence rate 
as Orthomin($2$) for $k\ge 2$, but smaller than that of Orthomin($1$). The latter systems
are related to the numerical solution of certain partial differential equations.
\end{abstract}

\begin{keyword} 
Orthomin($k$) \sep iterative methods
\MSC[2020] 65F10
\end{keyword}

\end{frontmatter}


\input{intro.tex}

\input{ortho1tok.tex}
\input{ellipsesexamples.tex}

\input{ortho_1v3.tex}
\input{conclusions.tex}
\input{appendixnumerics.tex}
\bibliography{ref}
\bibliographystyle{elsarticle-num}


\end{document}

%% file: intro.tex
\section{Introduction} 
\label{omin:sec:intro}

Originating with the work of Vinsome~\cite{Vinsome76}, Orthomin($k$)  ($k=1, 2, 3, \dots$) is  a family 
of iterative methods for solving linear systems of the form 
\begin{equation}
\label{omin:eq:gensys}
A x = b,
\end{equation}
where {\color{black}$A\in \C^{d\times d}$} is a nonsingular, possibly non-symmetric matrix, and $b\in\mathbb{C}^d$.
{\color{black}While Orthomin has received less attention compared to other iterative methods for non-symmetric systems,
it has also known some developments and extensions over time, e.g., to nonlinear systems~\cite{chen2001nonlinear}, and to
singular and inconsistent systems~\cite{MR1804663, abe2004variant, abe2008variant}.
It has also been applied to several problems involving fluid flows~\cite{li2005comparison, modak2006new, houzeaux2011extension}.}

Following~\cite{MR1990645}, Orthomin($k$) can be regarded as an incompletely orthogonalized, or truncated,
version of the \emph{Generalized Conjugate Residual} (GCR) method, similarly to the way the
\emph{Incomplete Orthogonalization Method} (IOM) is related to the 
\emph{Full Orthogonalization Method} (FOM), or quasi-GMRES to GMRES.
The main attraction of truncated Krylov-space methods lies in the fact that they lead to a fixed-term recurrence, thus 
the cost per iteration is fixed (or bounded). The downside is that the convergence rate is not guaranteed to improve by increasing 
the number of terms in the recurrence; to the best of our knowledge, this fact has  not been rigorously justified. 
We quote from~\cite{MR1474725} (p.34): ``Unfortunately, no stronger a priori bounds on the residual norm are known for 
Orthomin(2) applied to a general matrix whose field of values does not contain the origin although, in practice, 
it may perform significantly better than Orthomin(1).'' This is the main reason for which, in practice, GMRES is restarted after a
number of steps rather than the orthogonalization process being truncated, i.e., using quasi-GMRES.

The main goal of this article is to examine the behavior of Orthomin($k$) on a family of examples, and to show
essentially that for any $k>1$ there are examples of systems where Orthomin($k$) has the same convergence rate
as Orthomin($1$). This is consistent with the quote above; namely, it shows that, in absence of additional
assumptions on the matrix $A$, no a priori bounds can be found to show that Orthomin($k$) converges faster than
Orthomin($1$). 

\section{Brief background} 
\label{omin:sec:background}
In this section we describe Orthomin($k$) and give a brief background on the main known convergence result.
If $x_n$ is the $n^{\mathrm{th}}$ iteration  and $r_n=b-A x_n$ is the $n^{\mathrm{th}}$ residual,
the main idea is to find as $x_{n+1} = x_n+x$, with  the correction $x$ lying in a $k$-dimensional subspace~$V_k^{(n)}$
(or $(n+1)$-dimensional for $(n+1) < k$), so that the Euclidean norm of the next residual~$r_{n+1}=r_n-A x$ is minimized:
\be
\label{eq:rnortho}
\nnorm{r_{n+1}}^2 = \min_{x\in V_k^{(n)}}\nnorm{r_n-A x}^2\ .
\ee
The definition~\eqref{eq:rnortho} is equivalent to
\be
\label{eq:rdefproj}
r_{n+1} = r_n-\Pi_{A V_k^{(n)}} r_n\ ,
\ee
where $\Pi_V$ is the orthogonal projection on a subspace $V$. The algorithm generates a sequence of vectors
$p_0, p_1, p_2, \dots $, called the search directions, and for $n\ge k-1$ 
the space $V_k^{(n)}$ is generated by the last $k$ search directions $p_n, p_{n-1}, \dots, p_{n-k+1}$; 
for~\mbox{$n < k-1$} the space $V_k^{(n)}$ is simply $\mathrm{span}\{p_n, p_{n-1},\dots,p_0\}$. 
To give a precise formulation, for an initial guess  $x_0$ we initialize the residual and the
initial search direction by $p_0=r_0=b-A x_0$, and  the Orthomin($k$) iteration reads
\be
\label{eq:orthomincoeffa}
\lambda_n&=&\frac{\innprd{r_n}{A p_n}}{\innprd{A p_n}{A p_n}},\ \  \ 
\nu_n^{(j)}=\frac{\innprd{A r_{n+1}}{A p_{n-j+1}}}{\innprd{A p_{n-j+1}}{A p_{n-j+1}}},\\
\label{eq:orthominxr}
x_{n+1}&=&x_n+\lambda_n p_n\ , \ \ \ r_{n+1}=r_n-\lambda_n A p_n,\\
\label{eq:orthominp}
p_{n+1}&=&r_{n+1}-\sum_{j=1}^{\min(k-1,n+1)} \nu_n^{(j)} p_{n-j+1},
\ee 
for $j=1,\dots,\min(k-1,n+1)$. Here  $\innprd{u}{v}=\sum_{j=1}^d u_j
\overline{v_j}$ denotes  the inner product in $\mathbb{C}^d$, and
$\nnorm{u}\stackrel{\mathrm{def}}{=} \sqrt{\innprd{u}{u}}$.  The
coefficients $\lambda_n$ and $\nu_n^{(j)}$
in~\eqref{eq:orthomincoeffa} are defined so that  \be
\label{eq:aperp}
r_{n+1}\perp A p_n\ \   \mathrm{and}\ \ A p_{n+1}\perp A p_{n-j+1}\ ,\ \ \mathrm{for}\ \ 
j=1,\dots,\min(k-1,n+1)\ .  \ee 
An inductive argument shows that
$r_{n+1}\perp A p_{n-j+1}$ for $j=1,\dots,\min(k,n+1)$, and
hence~\eqref{eq:rnortho} holds.

Although {\color{black} of limited use} in practice, Orthomin($k$) could be thought of as attractive mainly for two  reasons. 
First, as with other truncated Krylov-space methods, Orthomin($k$) requires only one matrix-vector per iteration,
and the additional cost (per iteration) is~$O(k d)$~Flops; a maximum number of $k$ vectors need to be stored. 
Second, when  symmetric positive preconditioners are used
to produce a split preconditioning of Orthomin($k$), the preconditioned
iteration can be implemented without reference to the factors of the 
preconditioners. This is  a feature shared with preconditioned 
\emph{conjugate gradient} (CG), as shown by Elman~\cite{elmanthesis}, and it allows for matrix-free preconditioning.

In terms of convergence properties, Orthomin($k$) is guaranteed to converge, {\color{black}for any initial guess,} if 
the field of values\footnote{The \emph{field of values} or \emph{numerical range} 
of a complex matrix $A$ is defined as the set of complex numbers
$\op{F}(A) = \left\{ \innprd{A u}{u}\ :\ u\in \mathbb{C}^d,\ \norm{u}=1\right\}\ .$} 
of the matrix~$A$ does not contain~0. 
The precise convergence result and estimate shown below appears in~\cite{MR1474725} as Theorem~2.2.2, and was proved first 
in~\cite{MR694523} (see elso Elman~\cite{elmanthesis}) for matrices with positive
definite symmetric part. We recall the following result from~\cite{MR1474725}:
\begin{theorem}
\label{omin:th:om1descent}
Assume that $0\notin \op{F}(A)$ and $\delta=\mathrm{dist}(0, \op{F}(A))$. If $r_n$
is the~$n^{\mathrm{th}}$ residual in the Orthomin\textnormal{(}$k$\textnormal{)} iteration, then
\begin{equation}
  \label{omin:eq:om1descent}
  \nnorm{r_{n+1}}\le \nnorm{r_n} \sqrt{1-\frac{\delta^2}{\nnorm{A}^2}}\ ,
\end{equation}
where $\nnorm{A}$ is the $2$-norm of the matrix $A$.
\end{theorem}

We also recall from~\cite{MR1474725} the parallelism between Orthomin(1) and \emph{Steepest Descent} (SD) {\color{black}on} one hand, 
and between Orthomin(2) and CG, on the other. 
SD can only be used in connection to symmetric positive definite~(SPD) systems and 
has an iteration of the form~\eqref{eq:orthominxr} with the search direction given by  $p_n=r_n$, just like Orthomin(1). 
However, for SD the coefficient $\lambda_n$ is chosen so that 
$$e_{n+1}= e_n-\Pi_{\mathrm{span}\{r_n\}}^A e_n\ ,$$
where $\Pi_V^A$ is the projection on the subspace $V$ with respect to the $A$-inner product
$\innprd{u}{v}_A = \innprd{A u}{v}$. Consequently, the error estimates for steepest descent are
similar to the ones for Orthomin(1), and in practice the two methods converge comparably fast for SPD systems. Analogously, 
the sequence of search directions $p_0, p_1, \dots$ for CG follows a recursion that is similar
to Orthomin(2), except for in CG we have
$$e_{n+1}= e_n-\Pi_{\mathrm{span}\{p_n,p_{n+1}\}}^A e_n\ .$$
In addition, in the case of CG, the second set of orthogonality relations in~\eqref{eq:aperp}
is replaced by the $A$-orthogonality relation \mbox{$p_{n+1}\perp_A p_n$} (conjugate), whereas for Orthomin(2)
they read $A p_{n+1}\perp A p_n$.
Even though the superiority of CG over SD is well established and understood~\cite{MR1990645}, not the same can be said about
the relation of Orthomin(2) with Orthomin(1) for non-symmetric systems. 

The main contribution of this article is to show that  Orthomin($k$) {\bf does not perform better in general} (that is, for matrices $A$
that satisfy $0\notin \op{F}(A)$) than Orthomin($1$). In Section~\ref{sec:mainexample} we  consider matrices of the form
$A=I+\rho U$ with \mbox{$0<\rho<1$} and $U$ unitary. {\color{black} First we show that the convergence rate of  Orthomin($k$)
for such systems is less than or equal to $\rho$. Next we conjecture that, for certain examples -- all involving 
diagonal unitary matrices~$U$ -- the asymptotic convergence rate of  Orthomin($k$) is precisely $\rho$;} 
we support our conjecture with numerical evidence for $k\ge 2$ and we provide analytical arguments
for~\mbox{$k=1$} in Section~\ref{sec:o1convergence}, which forms the core of this article.
Prior to the analysis of the convergence rate of Orthomin(1), in Section~\ref{sec:ellipses} we give
examples of systems for which  Orthomin($j$), $j=2,\dots, k$ all achieve the same asymptotic convergence rate, but converge
faster than Orthomin(1).

%% file: ortho1tok.tex
\section{The main examples}
\label{sec:mainexample}
{\color{black}Throughout this article we denote by 
$\sigma(A)$ the spectrum of a matrix $A$. Furthermore, for $z\in \C$ and $\rho>0$ let 
$$\overline{\op{B}_{\rho}(z)} = \{w\in \C\: :\: |w-z|\le \rho\}.$$}
Consider the linear system
\be
\label{eq:sysmainex}
(I+\rho U)x = b\ ,
\ee
where $0<\rho<1$, $U\in {\color{black}\C^{d\times d}}$ is a unitary matrix, and $b\in\C^d$. 
Our goal is to assess the behavior/convergence of the ratios
\begin{equation}
\label{eq:qndef}
q_n=\frac{\nnorm{r^{(k)}_{n+1}}}{\nnorm{r^{(k)}_n}}\ ,
\end{equation}
where $r^{(k)}_n$ is the $n^{\mathrm{th}}$ residual in the Orthomin($k$) iteration. 

\subsection{An upper bound}
\label{ssec:upperbound}
The fact that
$q_n$ is bounded above by $\rho$ is a consequence of the following result.
\begin{theorem}
\label{th:convnormal}
Let $A\in {\color{black}\C^{d\times d}}$ be a normal matrix so that 
\be
\label{eq:speccond}
\sigma(A)\subseteq \overline{\op{B}_{\rho}(z_0)}
\ee
with \mbox{$0<\rho<\abs{z_0}$}.
The residuals $r_n^{(k)}$ obtained by applying the Orthomin$(k)$ iteration to 
the system~\eqref{omin:eq:gensys} satisfy
\be
\label{eq:convnormal}
\nnorm{r_{n+1}^{(k)}}\le \frac{\rho}{\abs{z_0}}\nnorm{r_n^{(k)}}\ .
\ee
\end{theorem}
%
\begin{proof}
Let $U=\rho^{-1}(A-z_0 I)$. Since $\sigma(A)\subseteq \overline{\op{B}_{\rho}(z_0)}$
we have $\sigma(U)\subseteq \overline{\op{B}_{1}(0)}$. Because $A$ is normal it follows that $U$ is also normal, hence
$\nnorm{U}_2\le 1$. 
If $p_0,p_1,\dots $ are the search directions of Orthomin($k$) we have
\bes
{r}_{n+1}^{(k)} = {r}_n^{(k)}-\Pi_{\mathrm{span}\{A p_n,\dots, A p_{n-j}\}} r_n^{(k)},
\ees
where $j=\min(n,k-1)$. Hence,
\bes
\nnorm{{r}_{n+1}^{(k)}} \le \nnorm{{r}_n^{(k)}- {v}}\ ,\ \ \forall {v}\in \mathrm{span}\{A p_n,\dots, A p_{n-j}\}\ .
\ees
Since
$
{r}_n^{(k)}\in \mathrm{span}\{{p}_n,\dots, {p}_{n-j}\}
$, we have 
\bes
A r_n^{(k)}\in \mathrm{span}\{{A p}_n,\dots, {A p}_{n-j}\}\ .
\ees
Therefore
\bes
\nnorm{{r}_{n+1}^{(k)}}\le \nnorm{{r}_n^{(k)}-z_0^{-1}A r_n^{(k)}} = 
\frac{\rho}{\abs{z_0}}\nnorm{U r_n^{(k)}}\le \frac{\rho}{\abs{z_0}}\nnorm{{r}_n^{(k)}}\ .
\ees
\end{proof}
Cf.~\cite{MR98b:47008}, if $A$ is normal, then $\op{F}({A})$ is equal to the convex hull  $\sigma({A})$.
Therefore,~\eqref{eq:speccond} is equivalent to
$\op{F}({A})\subseteq \overline{\op{B}_{\rho}(z_0)}$.
Hence, the general result~\eqref{omin:eq:om1descent} implies
\be
\label{eq:estclassic}
\nnorm{{r}_{n+1}^{(k)}}\le \sqrt{1-\frac{(\abs{z_0}-\rho)^2}{(\abs{z_0}+\rho)^2}}\:\nnorm{{r}_n^{(k)}}\  = 
\frac{2\sqrt{\rho/\abs{z_0}}}{1+\rho/\abs{z_0}}\:\nnorm{{r}_n^{(k)}}.
\ee
The bound~\eqref{eq:convnormal}, valid for normal operators only, is sharper than~\eqref{eq:estclassic}.

\subsection{Sharpness of the upper bound}
\label{ssec:sharpness}
To show that the estimate~\eqref{eq:convnormal} is sharp we consider the diagonal  matrices 
\be
\label{eq:Udef}
U=\mathrm{diag}\lbrack 1,\zeta_d,\zeta_d^2,\dots,\zeta_d^{d-1}\rbrack\ ,
\ee
where $\zeta_d=\exp(2\pi {\bf i}/d)$ is the primitive root of unity of order $d$. 
%
\begin{conjecture}
\label{conj:o1tokexamples}
For all $k\in\N$, there exists $d_k\in\N$ and  $0<\rho_k<1$ so that
for all $\rho\in(0,\rho_k)$ and $d\ge d_k$, the residuals $r_n^{(k)}$
obtained by applying the Orthomin$(k)$ iteration to 
the system~\eqref{eq:sysmainex} with $U$ of the form~\eqref{eq:Udef}  and
initial value $x_0=0$ satisfy
\be
\label{eq:o1tokexamples}
\lim_{n\to\infty}\frac{\nnorm{r_{n+1}^{(k)}}}{\nnorm{r_{n}^{(k)}}} = \rho\ .
\ee
\end{conjecture}
In this article we prove  Conjecture~\ref{conj:o1tokexamples} in the case when $k=1$ 
(see Theorem~\ref{th:ortho1proof} in Section~\ref{ssec:convfindim}). 
For $k\ge 2$, the numerical evidence in support of Conjecture~\ref{conj:o1tokexamples} is quite strong, 
as shown  in {\color{black}\ref{sec:numresults}}. A consequence of Conjecture~\ref{conj:o1tokexamples} is that
for a given $k\in \N$ we can find linear systems for which all of Orthomin($j$), $j=1,\dots, k$, achieve the
same convergence rate. {\color{black}This establishes why there is no {\bf general} result, e.g.,  
valid for all matrices $A$ for with $0\notin \op{F}(A)$, showing that Orthomin($k$) converge faster than
Orthomin(1)}.  Naturally, for any system in $\C^d$, Orthomin($d$) will converge in at most $d$ steps; {\color{black}
  therefore, it is necessary that $d_k> k$ in order for~\eqref{eq:o1tokexamples} to make sense. This is not a sufficient condition, 
  as seen from Proposition~\ref{prop:o1d2}, where it is shown that $d_1\ge 3$ is necessary for~\eqref{eq:o1tokexamples} to hold.
  In addition, for Orthomin($1$), Conjecture~\ref{conj:hull} suggests that ${\rho}_k=\cos(\pi/d_k)$, which represents the radius of the
  circle inscribed in the regular polygon formed by the roots of unity or order $d_k$. If Orthomin($1$) serves as a guide, 
  this value for ${\rho}_k$ may be a good choice for Orthomin($k$) as well.
}


\subsection{Connection with numerical partial differential equations}
\label{sssec:pde}
Systems of the form~\eqref{eq:sysmainex}  arise naturally in the
numerical solution of partial differential equations (PDEs). Consider the steady-state advection-reaction-diffusion
equation on $[0,2 \pi]$
\be
\label{eq:secorderpde}
-a u''(x) + b u'(x) + c u(x) = f(x)\ ,\ \ a>0,\  c\ge 0,\  b\in \R\ ,
\ee
with periodic boundary conditions $u(0)=u(2 \pi),\ u'(0)=u'(2\pi)$. To obtain a discretization of~\eqref{eq:secorderpde}
we proceed as follows: set $x_j=j h$, $j=0,1,\dots, d$, $h=2\pi/d$, be a uniform grid
(we identify $x_0$ with $x_d$, $x_{-1}$ with $x_{d-1}$, and $x_{1}$ with $x_{d+1}$), and replace the 
derivatives in~\eqref{eq:secorderpde} with the usual centered difference formulas
$$
-u''(x_j)\approx\frac{2 u(x_j)-u(x_{j-1})-u(x_{j+1})}{h^2}\ ,\ \ \ u'(x_j)\approx\frac{u(x_{j+1})-u(x_{j-1})}{2 h}\ .
$$
The  resulting discretization\footnote{This 
particular discretization is not appropriate for advection-dominated problems.}
 is a linear system of type~\eqref{omin:eq:gensys}: 
$A$ is a normal matrix with orthogonal eigenvectors $\chi^{(k)}\in \C^d$  and corresponding eigenvalues $\lambda_k$ given by
$$\chi^{(k)}_j = \exp({\bf i} k j h)\ ,\ \ \lambda_k = -\frac{2 a}{h^2}\cos(k h)+ 
{\bf i} \frac{b}{h}\sin(k h) + c+\frac{2 a}{h^2}\ .$$
The eigenvalues lie on an ellipse with semi-axes $2 a/h^2$ and $b/h$; when $2 a=h b$ this is a circle
of radius $b/h$. After further rescaling, the system can be brought to  the form~\eqref{eq:sysmainex}.
However, as will be shown in Section~\ref{sec:ellipses}, this example is quite relevant to 
the convergence study of Orthomin$(k)$ also when $2 a\ne h b$.

%% file: ellipsesexamples.tex
\section{Further examples: normal matrices with spectra on ellipses}
\label{sec:ellipses} So far we have examined the systems~\eqref{eq:sysmainex}, and  we conjectured
 that for any $k\in\N$ we can find operators $U$ of the form~\eqref{eq:Udef} so that for 
all  $1\le j\le k$, Orthomin($j$) achieves an asymptotic convergence rate equal to $\rho$. 
After a trivial rescaling, we restate Conjecture~\ref{conj:o1tokexamples} in the  following way:
for any circle $\op{C}$ of center $z_0$ and radius $\rho$ satisfying $0<\rho<\abs{z_0}$
there exists a normal matrix $A$ whose spectrum lies
on~$\op{C}$ so that for all $1\le j\le k$, the Orthomin($j$) iteration applied to the system~\eqref{omin:eq:gensys}
with $b=(1,1,\dots,1)^T$ and zero initial guess has an asymptotic convergence rate of~$\rho/\abs{z_0}$.

In this section we show numerical evidence suggesting that if we replace
the circle~$\op{C}$ with a non-circular ellipse~$\op{E}$ in the example above,
all Orthomin($j$) with $k\ge 2$ achieve {\bf the same} asymptotic convergence rate~$\rho_{\op{E}}$,
which is smaller than the asymptotic convergence rate of Orthomin(1). For the exact formulation see
Conjecture~\ref{conj:ellipsedisc}. We remark that the discretized numerical PDE from Section~\ref{sssec:pde} 
is an example of precisely such a system.

In order to make the examples very specific we first describe an ellipse~$\op{E}$ by its semi-axes $\alpha>0$ and $\beta>0$, 
the angle $\theta\in \R$ between its axes and the coordinate axes, and the position $u\in \C$  of its center:
\begin{equation}
  \label{eq:ellipsedef}
  \op{E}=\left\{u+e^{{\bf i} \theta} (\alpha \cos \gamma + {\bf i} \beta \sin \gamma)\ :\ \gamma\in [0,2\pi]\right\}\ .
\end{equation}
It is assumed that $0$ does {\color{blue}\bf not} lie on or inside  $\op{E}$. 
For $d\in \N$ we consider the numbers $\mu_j\in \op{E}$ defined 
as
\begin{equation}
  \label{eq:muellipse}
  \mu_j=u+e^{{\bf i} \theta} \left(\alpha \cos \frac{2\pi j}{d} +
     {\bf i} \beta \sin \frac{2\pi j}{d}\right)\ ,\ \ j=1,\dots, d\ .
\end{equation}
As before, we associate to $\op{E}$ a linear operator
$$A_{\op{E},d} \stackrel{\mathrm{def}}{=} \mathrm{diag}[\mu_1,\dots,\mu_d]\ .$$

\begin{conjecture}
\label{conj:ellipsedisc}
For any ellipse $\op{E}$  there exists a number $\rho_{\op{E}}\in (0,1)$
so that the following hold:\\
\begin{itemize}
\item[\textnormal{(}i\textnormal{)}] For all $k\in \N$ with  $k\ge 2$, there exists $d_k\in \N$ so that for $d\ge d_k$ 
  the ratio $q_n=\nnorm{r_{n+1}^{(k)}}/\nnorm{r_{n}^{(k)}}$ of the residual-norm obtained by applying the Orthomin$(k)$ iteration
  with zero initial guess to the system 
  \begin{equation}
    \label{eq:linsysellipse}
    A_{\op{E},d} x = (1,1,\dots,1)^T 
  \end{equation}
  satisfies
  \begin{equation}
    \label{eq:qntorho}
    \lim_{n\to\infty} q_n = \rho_{\op{E}}\ .
  \end{equation}
\item[\textnormal{(}ii\textnormal{)}] If the ellipse is not circular, then $\rho_{\op{E}}$ is smaller than the asymptotic convergence rate of Orthomin$(1)$.
\end{itemize}
\end{conjecture}

Two facts are notable about the behavior of Orthomin$(k)$ for the systems in Conjecture~\ref{conj:ellipsedisc}.
First, it is remarkable that the ratios $q_n$ converge at all; indeed, we show that for~$k=1$ the sequence~$\{q_n\}_{n\in \N}$ 
is convergent regardless of the choice of the numbers $\mu_1,\dots,\mu_d$,
 but for $k\ge 2$ the sequence $\{q_n\}_{n\in \N}$  may not be monotone, and is not expected to converge in general.
The second interesting fact is that all Orthomin$(k)$ with $k\ge 2$ achieve the same asymptotic convergence rate for sufficiently large~$d$. 
Moreover, numerical experiments show that $q_n$ converges to the same limit~$\rho_{\op{E}}$ even for a random initial guess and right-hand side $b$.
However, in spite of the fact that $\rho_{\op{E}}$ seems to be intimately related to the ellipse, currently we do not
 understand the nature of this connection, i.e., how to compute $\rho_{\op{E}}$ using only information about~$\op{E}$.

We conclude this section by showing numerical evidence in support of Conjecture~\ref{conj:ellipsedisc}. For numerical experiments we have selected an ellipse
in general position (not aligned with the coordinate axes) with  $\alpha=2$, $\beta=1$, $u=2+{\bf i}$, and $\theta=\pi/6$.
For $d=128$ we solved the system~\eqref{eq:linsysellipse} using Orthomin$(k)$ with $k=1, 2, 3, 4, 10$. In Figure~\ref{fig:qn128ellipse} 
we plot the ratios~$q_n$ for each of the solves. The data strongly suggests that for $k=2, 3, 4, 10$ we have
$$q_n\to \rho_{\op{E}}\approx 0.6891227\ .$$
This approximate value (up to the first eight digits)
was also obtained when solving~\eqref{eq:linsysellipse} with random right-hand side and initial guess. 
In the particular case of
Orthomin(1), we know that  $q_n$ is convergent (and  increasing): numerically we find that  $\lim q_n\approx 0.7902$.


\begin{figure}[!htb]
\begin{center}
\includegraphics[width=\textwidth]{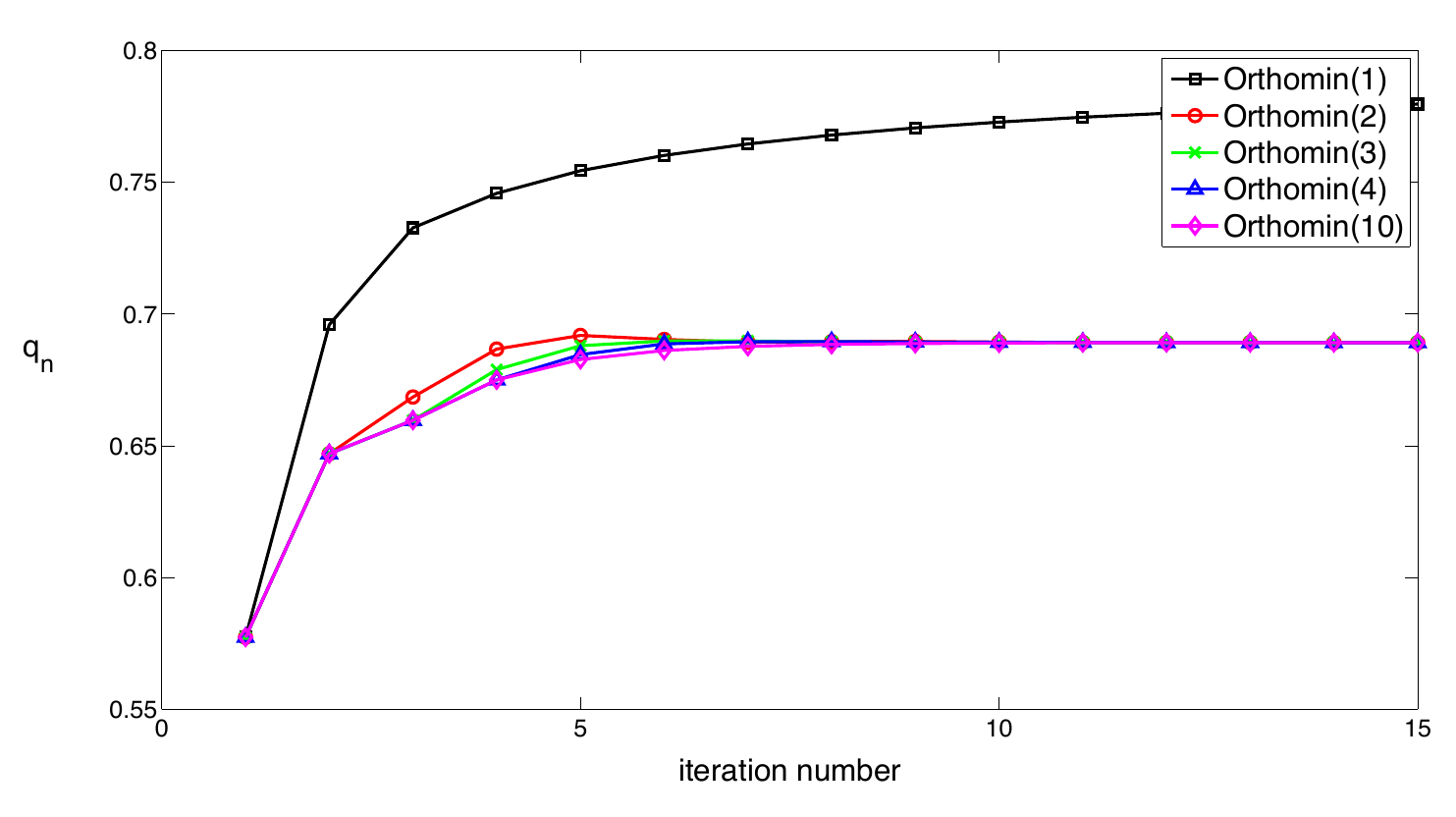}    
\caption{The comparative residual norms for Orthomin(k) ($k=1,2,3,4,5,10$): for Orthomin$(1)$
$q_n$ exceeds $0.7902$, but for $k=2, 3, 4, 5, 10$ we note a convergence of $q_n$ to a value near $0.6891227$.} 
\label{fig:qn128ellipse}
\end{center}
\end{figure}

%% file: ortho_1v3.tex
\section{Convergence analysis for Orthomin(1)}
\label{sec:o1convergence}
The main objective of  this section is to prove 
Conjecture~\ref{conj:o1tokexamples} 
for $k=1$.
In Section~\ref{ssec:monotonicity} we show that the sequence $\{q_n\}_{n\in\N}$
is increasing and bounded. After stating in Section~\ref{ssec:abszetaisone} 
a few technical results, we discuss in Section~\ref{ssec:nonconvergence} examples 
when $q_n$ does not converge to~$\rho$. The behavior of~$q_n$
for  two-dimensional systems is presented in Section~\ref{ssec:sis2}. 
In Section~\ref{ssec:convfindim} we prove Conjecture~\ref{conj:o1tokexamples} 
for $k=1$. 

We consider matrices of the form 
\be
\label{eq:AgeneralO1}
A=\mathrm{diag}[\mu_1,\dots, \mu_d]\ ,
\ee
with $\mu_1,\dots, \mu_d\in\C$ nonzero complex numbers. Since we are interested
in the evolution of the residuals, we retain only the recursive equation from Orthomin(1) 
that produces the residual $r_n=r_n^{(1)}$:
\begin{equation}\label{eq:rdefo1}
r_{n+1}=r_n-\Pi_{A r_n}r_n\ ,
\end{equation} 
with $r_0\in\C^d$ being chosen arbitrarily. Recall from~\eqref{eq:orthomincoeffa} and~\eqref{eq:orthominxr} that
$$
\lambda_n=\frac{(r_n, Ar_n)}{(A r_n, A r_n)}\ ,
\ \ r_{n+1}=r_n-\lambda_nAr_n\ .
$$
Let $r_n=(r_n^1, \dots r_n^d)$ be the coefficients of $r_n$. We consider the 
finite probability measure supported at $1,\dots, d$ with weights
proportional to $|r_n^1|^2$, \dots, $|r_n^d|^2$.
We will refer to it as the $r_n$-measure, and use the subscript $n$ to denote it.  
For instance, the expected value of a vector $\xi=(\xi_1,\dots, \xi_d)$ 
with respect to this measure is
$$\EE_n(\xi):=\frac{\sum_{k=1}^d\xi_k|r_n^k|^2}{\sum_{k=1}^d|r_n^k|^2}\ .$$
Since $r_{n+1}=r_n-\lambda_nAr_n$ has coefficients
$r_{n+1}^k=(1-\lambda_n\mu_k)r_n^k$,  the following change of variable formula holds:
\begin{equation}
\label{eq:cofv}
\EE_{n+1}(\xi)=\frac{\EE_n(\xi|1-\lambda_n\mu|^2)}
{\EE_n(|1-\lambda_n\mu|^2)}\ ,
\end{equation}
where
$\mu=(\mu_1,\dots, \mu_d)$ is the vector of eigenvalues of $A$.
In particular, 
\begin{equation}
\label{eq:lambdan}
\lambda_n
=\frac{\sum_{k=1}^d\bar\mu_k|r_n^k|^2}{\sum_{k=1}^d|\mu_k|^2|r_n^k|^2}=
\frac{\EE_n(\bar\mu)}{\EE_n(|\mu|^2)}\ .
\end{equation}
\subsection{Monotonocity of $q_n$}
\label{ssec:monotonicity}
We begin with a technical result.
\begin{lemma}
\label{lma:pearson}
Let $\xi$ a complex-valued random variable with finite moments up to order $4$ satisfying 
the identity $\EE(\xi)=\EE(|\xi|^2)$. The following inequality then holds:
\begin{equation}
\EE(|\xi|^2)\:\EE(|1-\xi|^2)\:\EE(|\xi|^2|1-\xi|^2)
\geq |\EE(\xi|1-\xi|^2)|^2\ .
\end{equation}
\end{lemma}   
\begin{proof}
First of all, we remark that if $\xi$ satisfies the condition stated in the Lemma,
then so does $1-\xi$. Thus, the situation is 
symmetric in $\xi$ and $1-\xi$.  
\\\indent
Let $\theta\in\R$ such that $\EE(\xi|1-\xi|^2)=e^{i\theta}|\EE(\xi|1-\xi|^2)|$. 
Consider the function  
\begin{equation}
f:\R\to\R,\quad 
f(t):=\operatorname{Var}(t(1-\xi)+e^{i\theta}\bar\xi(1-\xi))\ ,
\end{equation}
where Var$(\xi) = \EE(\abs{\xi}^2)-\abs{\EE(\xi)}^2$ denotes the variance of a random variable $\xi$.
By opening up the parenthesis inside the expected value, we obtain
\[\begin{split}
f(t)
&= 
t^2\{\EE(|1-\xi|^2)-|\EE(1-\xi)|^2\}
+2t|\EE(\xi|1-\xi|^2)|
+\EE(|\xi|^2|1-\xi|^2)
\\
&=
t^2\EE(|\xi|^2)\EE(|1-\xi|^2)
+2t|\EE(\xi|1-\xi|^2)|
+\EE(|\xi|^2|1-\xi|^2).
\end{split}\]
The second equality follows fom 
a manipulation of the coefficient of $t^2$ which takes into account
the fact that $\EE(\xi)=\EE(|\xi|^2)$. 
This shows that $f(t)$ is a real valued quadratic form. 
The fact that it is a positive definite quadratic form follows from the fact 
that the variance of a random variable is always a positive number. 
Therefore, $f(t)$ has negative discriminant:
\[
|\EE(\xi|1-\xi|^2)|^2-\EE(|\xi|^2)\EE(|1-\xi|^2)\EE(|\xi|^2|1-\xi|^2)\leq 0\ ,
\]
which completes the proof.
\end{proof}
We should point out that in the case when $\xi$ is real valued (which is not the case here), 
the statement of Lemma~\ref{lma:pearson}
can be  reduced to Pearson's inequality~\cite{MR0010938} (see also~\cite{MR1791507}) between the skewness $\tau$ and 
the kurtosis $\kappa$ of a distribution:  $$\kappa-\tau^2-1\geq 0.$$ 
We do not give a proof of this fact, as it is of no relevance to the rest 
of the paper.  
We now show that the sequence $q_n$ is increasing and bounded.
\label{ssec:monotonicityqn}
\begin{proposition}
\label{prop-monotone}
If $r_n$ is given by~\eqref{eq:rdefo1} and $A$ is defined as 
in~\eqref{eq:AgeneralO1}, then $q_n$ is increasing  and bounded between  $0$ and $1$.
\end{proposition}
\begin{proof}
We will use the measure-theoretic notation:
\bes
q_n^2&=&\frac{\nnorm{r_{n+1}}^2}{\nnorm{r_n}^2}
=\EE_n(|1-\lambda_n\mu|^2)
=\EE_n(1+|\lambda_n|^2|\mu|^2-\lambda_n\mu-\bar\lambda_n\bar\mu)
\\
&=&1+|\lambda_n|^2\EE_n(|\mu|^2)-\lambda_n\EE_n(\mu)
-\bar\lambda_n\EE_n(\bar\mu)
\stackrel{\eqref{eq:lambdan}}{=}
1-\frac{|\EE_n(\mu)|^2}{\EE_n(|\mu|^2)}\ .
\ees
We compare  $1-q_{n+1}^2$ and $1-q_n^2$. For the latter, we use 
the change of variable formula~\eqref{eq:cofv}:
\[
1-q_{n+1}^2
=\frac{|\EE_{n+1}(\mu)|^2}{\EE_{n+1}(|\mu|^2)}
=\frac{|\EE_n\left(\mu|1-\lambda_n\mu|^2\right)|^2}
{\EE_n\left(|1-\lambda_n\mu|^2\right) 
\EE_n\left(|\mu|^2|1-\lambda_n\mu|^2\right)}\ .
\]
We can re-write this as 
$$1-q_{n+1}^2
=\frac{|\EE_n(\xi|1-\xi|^2)|^2}
{\EE_n(|1-\xi|^2)\EE_n(|\xi|^2|1-\xi|^2)}\ ,$$ 
with $\xi=\lambda_n\mu$. By construction, 
$$\EE_n(\xi)=\EE_n(|\xi|^2)=\frac{|\EE_n(\mu)|^2}{\EE_n(|\mu|^2)}\ ,$$ 
hence
we can apply the result of Lemma~\ref{lma:pearson} to $\xi$: 
\[
1-q_n^2=\EE_n(|\xi|^2)\geq 
\frac{|\EE_n(\xi|1-\xi|^2)|^2}
{\EE_n(|1-\xi|^2)\EE_n(|\xi|^2|1-\xi|^2)}=1-q_{n+1}^2,
\]
hence $q_n\leq q_{n+1}$.
\end{proof}

Note that we can think of $q_n$ as measuring the dispersion
of the random variable $\mu$ relative to the $r_n$-measure: variance about the mean
divided by average size. The monotonicity of $q_n$ reflects the fact that $\mu$
becomes increasingly more uniformly distributed relative to the $r_n$-measures.

{\color{black}We remark that Proposition~\ref{prop-monotone} holds for all  normal (non-singular) matrices.
Indeed, if $A$ is normal, then we can write  $A=U D U^*$ with $U$ unitary and $D$ diagonal. Consider the
change of variable $\tilde{x} = U^* x $ and $\tilde{b}=U^* b$. Then~\eqref{omin:eq:gensys} is
equivalent to the system $D \tilde{x} = \tilde{b}$, and the residuals are linked via the relation
$$
\tilde{r}_n = \tilde{b}-D\tilde{x}_n = U^* (b - U D U^* x_n) = U^* r_n.
$$
Furthermore, the coefficients $\lambda_n$ satisfy:
$$
\lambda_n = \frac{(r_n, Ar_n)}{(A r_n, A r_n)} = 
\frac{(U \tilde{r}_n, U D U^* {r}_n)}{(U D U^*{r}_n, U D U^*{r}_n)}
=
\frac{( \tilde{r}_n, D  \tilde{r}_n)}{( D \tilde{r}_n, D \tilde{r}_n)} = \tilde{\lambda}_n,
$$
proving that~$\tilde{r}_n$ is also the result of applying Orthomin(1) to the transformed system. 
Thus we have
$$
q_n = \frac{\|r_{n+1}\|}{\|r_{n}\|} = \frac{\|\tilde{r}_{n+1}\|}{\|\tilde{r}_{n}\|},
$$
and it follows from Proposition~\ref{prop-monotone} that $q_n$ is increasing.

We also point out that the result in
Proposition~\ref{prop-monotone} is not new; in fact, Orthomin(1) is identical to GMRES(1), and it was 
shown in~\cite{MR2599773} that 
$$
\frac{\|r_n\|}{\|r_{n-1}\|} \le \frac{\|r_{n+1}\|}{\|r_{n}\|},
$$
where $r_{n}$ is the $n^{\mathrm{th}}$ residual of the restarted GMRES($k$). Hence, Proposition~\ref{prop-monotone}
is a particular case of Theorem~5 in~\cite{MR2599773}. However, we believe our proof offers an alternative
argument leading to the particular result of interest to the behavior of Orthomin(1).}

\subsection{The case $\mu_k=1+\rho\zeta_k$, $|\zeta_k|=1$, and $r_0\in\C^d$ arbitrary}
\label{ssec:abszetaisone}
In this section we assume that $A$ is of the form
$$A=I+\rho U,\ \  \ U=\operatorname{diag}[\zeta_1,\dots, \zeta_d],$$ 
with $0<\rho<1$ and
$|\zeta_1|=\dots=|\zeta_d|=1$. Also, we keep
$r_0\in\C^d$ arbitrary unless otherwise specified.
We introduce the following quantities, for $n\geq 0$:
\begin{equation}
\omega_n=\frac{(U r_n, r_n)}{(r_n, r_n)},\quad
\tau_n=\frac{1-\lambda_n}{\rho\lambda_n}\ .
\end{equation}
Note that the coefficients of $r_{n+1}$ are related to those of $r_n$ as follows
\begin{equation}
r_{n+1}^k=(1-\lambda_n\mu_k)r_n^k
=\rho\lambda_n(\tau_n-\zeta_k)r_n^k,
\end{equation}
and the change of variable formula becomes
\begin{equation}
\label{eq:cofv2}
\EE_{n+1}(\xi)=
\frac{\EE_n(\xi|\tau_n-\zeta|^2)}{\EE_n(|\tau_n-\zeta|^2)}\  .
\end{equation}
\begin{lemma} For $n\geq 0$ we have
\begin{equation}
\lambda_n=\frac{1+\rho\bar\omega_n}{1+\rho^2+2\rho\Re\omega_n},\quad
\tau_n=\frac{\omega_n+\rho}{\rho\bar\omega_n+1},\quad
q_n^2=\rho^2\frac{1-|\omega_n|^2}{1+\rho^2+2\rho\Re\omega_n}\ ,
\end{equation}
where $\Re z$ denotes the real  part of a complex number $z$.
\end{lemma}
\begin{proof}
Let $\zeta=(\zeta_1,\dots, \zeta_d)$. 
Clearly, $\omega_n=\EE_n(\zeta)$.
Since $\mu=1+\rho\zeta$, we have
\[\lambda_n
=\frac{\EE_n(1+\rho\bar\zeta)}{\EE_n(|1+\rho\zeta|^2)}\ .\]
The formula for $\lambda_n$ then follows from
the fact that $\EE_n(1+\rho\bar\zeta)=1+\rho\bar\omega_n$,
and $\EE_n(|1+\rho\zeta|^2)=1+\rho^2+2\rho\Re\omega_n$. 
Next, the formula of $\tau_n$ is a direct consequence of the formula of $\lambda_n$.
Finally, 
\begin{eqnarray*}
q_n^2&=&1-\frac{|\EE_n(\mu)|^2}{\EE_n(|\mu|^2)} = 
1-\frac{1+\rho^2|\omega_n|^2+2\rho\Re\omega_n}{1+\rho^2+2\rho\Re\omega_n}
=\frac{\rho^2(1-|\omega_n|^2)}{1+\rho^2+2\rho\Re\omega_n}\ .
\end{eqnarray*}
\end{proof}
\begin{proposition}
\label{th:convequiv}
For $n\geq 0$ we have $|\omega_n|\leq 1$ and $ 0\leq q_n\leq\rho$.
Moreover, the following statements are equivalent:
\begin{equation}
\label{eq:convequiv}
{\bf (a)}\lim_{n\to\infty}q_n=\rho\:;
\quad
{\bf (b)} \lim_{n\to\infty} \omega_n=-\rho\:;
\quad
{\bf (c)} \lim_{n\to\infty}\lambda_n=1\:;
\quad
{\bf (d)} \lim_{n\to\infty}\tau_n= 0\:.
\end{equation}
\end{proposition}
\begin{proof}
The bound $|\omega_n|\leq 1$ follows from  $\|U\|\leq 1$.
The fact that $q_n$ is increasing has been proved in the previous section, and
the bound $q_n\leq\rho$ is a direct consequence of Theorem \ref{th:convnormal}.
Since  $\lambda_n$, $\tau_n$,  $q_n$ are continuous functions of $\omega_n$,
the statement (b) clearly implies all the others.
We also have (a) $\Rightarrow$ (b)
since 
$$\frac{1-|\omega|^2}{1+\rho^2+2\rho\Re\omega}\leq 1,$$ with equality 
for $\omega=-\rho$.
Similarly (d) $\Rightarrow$ (b) since $\tau_n$ has bounded denominator.
Finally,
$$1-\lambda_n=\frac{\rho(\rho+\omega_n)}{1+\rho^2+2 \rho \Re \omega_n}\ .$$
Since the denominator is bounded, 
{\color{black}
\begin{align*}
\lim_{n\to \infty} \lambda_n = 1 \ \  \mathrm{implies}\ \  \lim_{n\to \infty}\omega_n = -\rho, 
\end{align*}
}
showing that
(c)$\Rightarrow$ (b).
\end{proof}
{\color{black}
In addition to the quantities above, we define the ratios
\begin{equation}
\label{eq:defomegank}
\omega_{n,j}:=
\frac{(U^j r_n, r_n)}{(r_n, r_n)}
=\EE_n(\zeta^j),\quad
 j\geq 0\ ,
\end{equation}
which will play an important role in the convergence argument  (Section~\ref{ssec:convfindim}). 
These satisfy a recurrence relation:
\begin{proposition}
\label{th:omegarec}
For $n\geq 0$ we have the following recurrence relation
\begin{equation}
\label{eq:omegarec}
\omega_{n+1,j}
=\frac{(1+|\tau_n|^2)\omega_{n,j}-\tau_n\omega_{n,j-1}-\bar \tau_n\omega_{n,j+1}}
{1+|\tau_n|^2-2\Re(\bar \tau_n\omega_n)} ,\quad j\geq 1\ .
\end{equation}
\end{proposition} 
\begin{proof}
It helps to think of $\omega_{n,j}$ as moments of the $r_n$-distribution, since by definition
$$\omega_n,j=\mathbb E_n[\zeta^j].$$
Clearly, $\omega_{n,0}=1\ \ \mathrm{and}\ \ \omega_{n,1}=\omega_n.$
Using the change of variable formula (\ref{eq:cofv2}), we get:
\[
\omega_{n+1,j}=\EE_{n+1}(\zeta^j)
=\frac{\EE_n(\zeta^j|\tau_n-\zeta|^2)}{\EE_n(|\tau_n-\zeta|^2)}
=\frac{\EE_n\{\zeta^j(1+|\tau_n|^2-\tau_n\bar\zeta-\bar \tau_n\zeta)\}}
{\EE_n\{1+|\tau_n|^2-\tau_n\bar\zeta-\bar \tau_n\zeta\}}\ ,
\]
and the result follows.
\end{proof}
}
%

\subsection{Non-convergence to $\rho$}
\label{ssec:nonconvergence}
Let $\Hull(\zeta_1,\dots, \zeta_d)$ denote the convex hull of $\zeta_1,\dots, \zeta_d$. 
This is a compact convex subset of $\C$. 
Since
$$\omega_n=\frac{\sum_{k=1}^d\zeta_k|{\color{black}r_n^k}|^2}{\sum_{k=1}^d|{\color{black}r_n^k}|^2}\in\Hull(\zeta_1,\dots, \zeta_d)\ ,$$
the sequence $\omega_n$ cannot converge to $-\rho$ unless
$-\rho\in \Hull(\zeta_1,\dots, \zeta_d)$.
Since the statements {\color{black}$\lim_{n\to \infty}\omega_n=-\rho$ and $\lim_{n\to \infty}q_n=\rho$} are
equivalent, we have the following.
\begin{proposition}
\label{prop:nonconv}
Assume $-\rho\notin\Hull(\zeta_1,\dots, \zeta_d)$. Then
$\lim_{{\color{black}n\to\infty}}q_n\neq \rho$.
\end{proposition}
\begin{corollary}
\label{cor:arccos}
Assume that $\rho\in (0,1)$ is arbitrary, and
$|\theta_k|<\pi-\arccos(\rho)$, for $1\leq k\leq d$. 
If
 $\zeta_k=\exp(i\theta_k)$,  then
$\lim_{{\color{black}n\to \infty}} q_n\neq\rho$. 
\end{corollary}
\begin{proof}
The angles are chosen so that $\Re(\zeta_k)> -\rho$.
This ensures $-\rho\notin\Hull(\zeta_1,\dots, \zeta_d)$, 
and  the previous Proposition applies.
\end{proof}

Figures~\ref{fig:nonconv1} and~\ref{fig:nonconv2}
illustrate the context of Corollary~\ref{cor:arccos}:~$q_n$ does not converge to $\rho$,
and  $\omega_n$ does not converge to $-\rho$.
We end this section with a sharpened version of Conjecture~\ref{conj:o1tokexamples} for $k=1$:
\begin{conjecture}
\label{conj:hull} For Orthomin$(1)$, 
if $-\rho\in\Hull(\zeta_1,\dots, \zeta_d)$, then {\color{black}$$\lim_{n\to \infty}q_n=\rho.$$}
\end{conjecture}

\begin{figure}[!htb]
\begin{center}
\includegraphics[width=\textwidth]{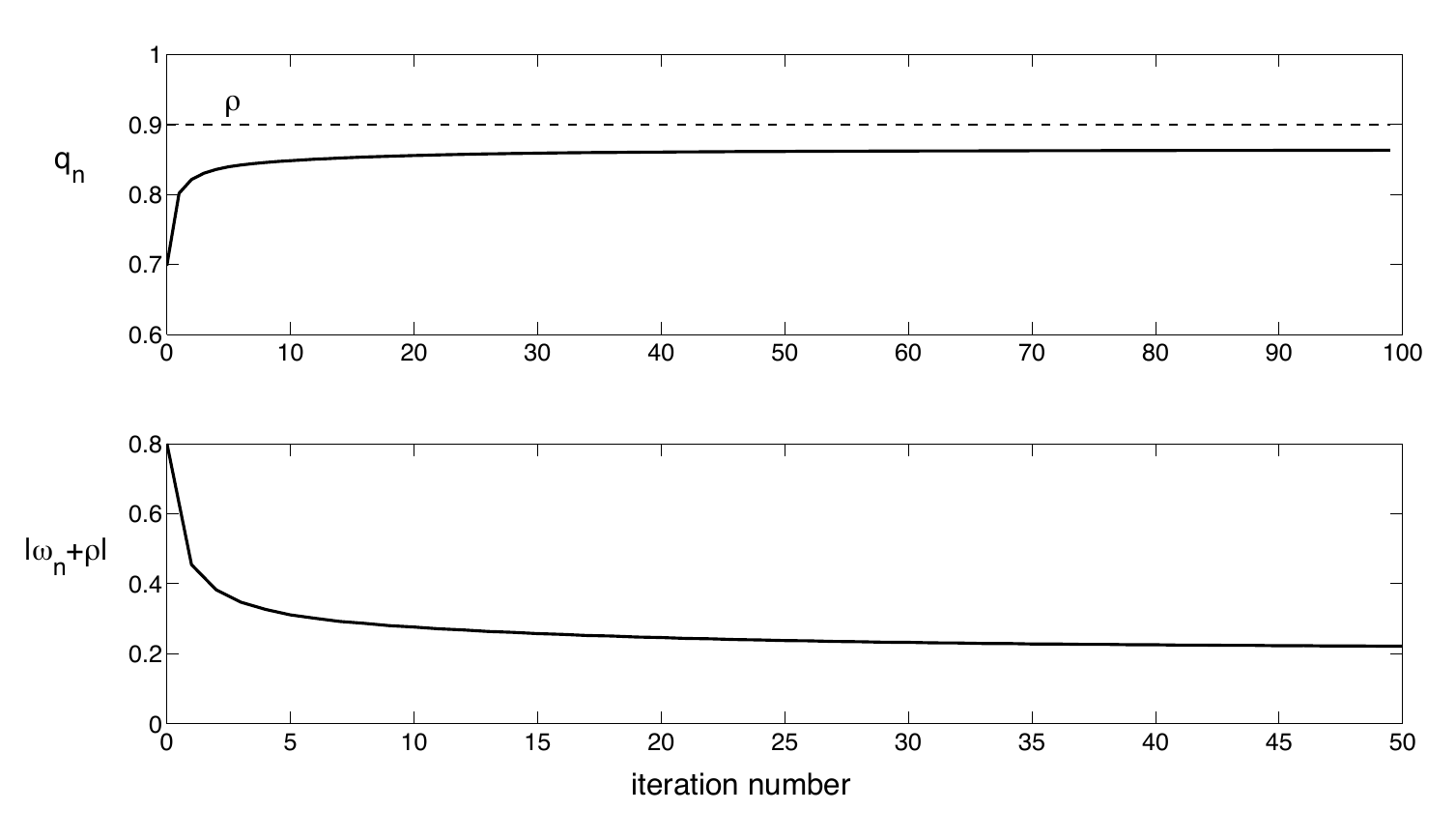}    
\caption{This is an example where $-\rho$ does not belong to $\Hull(\zeta_1,\dots, \zeta_d)$: $\rho=0.9,\ d=15$.} 
\label{fig:nonconv1}
\end{center}
\end{figure}

\begin{figure}[!htb]
\begin{center}
\includegraphics[width=\textwidth]{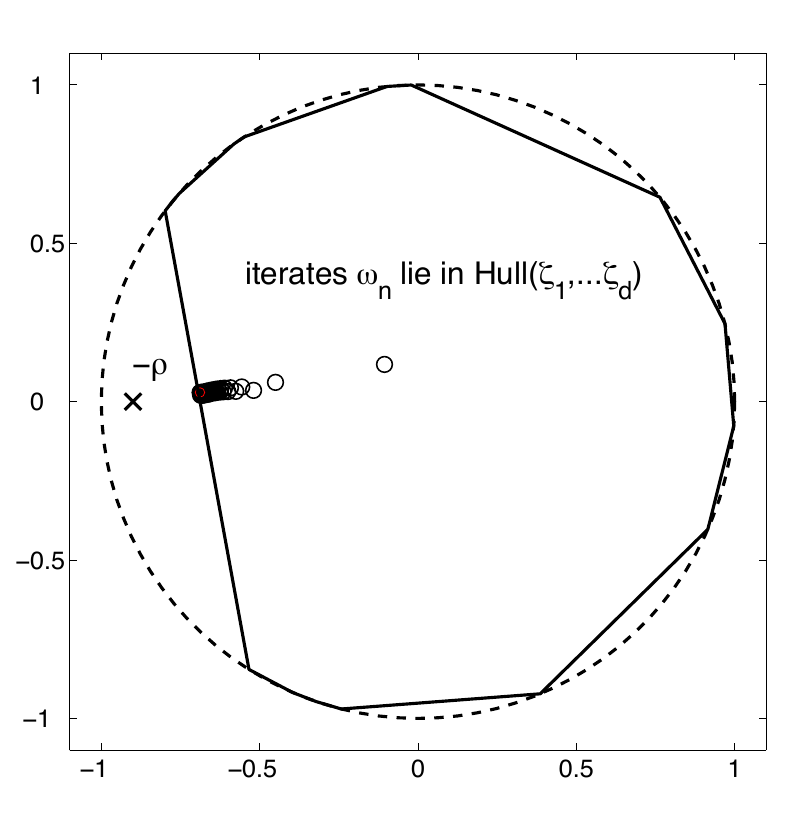}    
\caption{The case when $-\rho$ does not belong to $\Hull(\zeta_1,\dots, \zeta_d)$: $\rho=0.9,\ d=15$.}
\label{fig:nonconv2}
\end{center}
\end{figure}


\subsection{The case d=2}
\label{ssec:sis2}
Surprisingly, this case is not completely trivial either.
\begin{proposition}
\label{prop:o1d2}
Assume $d=2$ and the initial vector $r_0\in\C^2$ is arbitrary, with non-zero entries.
Then $q_n$ is a constant depending on $r_0$, while $\omega_n$ is a periodic
sequence with period $2$. The convergence~\eqref{eq:o1tokexamples}
for $k=1$ does not hold in this case.
\end{proposition}
\begin{proof}
With a rotation, we may assume $\zeta_1=1$ and $\zeta_2=\zeta$ is arbitrary.
Then $\mu_1=1+\rho$ and $\mu_2=1+\rho\zeta$.
We have
$$\lambda_0
=\frac{(r_0,Ar_0)}{(Ar_0,Ar_0)}=\frac{\bar\mu_1 |r_0^1|^2+\bar\mu_2|r_0^2|^2}
{\|Ar_0\|^2},\quad
\omega_0=\frac{|r_0^1|^2+\zeta|r_0^2|^2}{|r_0^1|^2+|r_0^2|^2}\ ,$$
therefore
\[
1-\lambda_0\mu_1=
\frac{-\rho(1-\zeta)\bar\mu_2|r_0^2|^2}
{\|Ar_0\|^2}\ ,
1-\lambda_0\mu_2
=\frac{\rho(1-\zeta)\bar\mu_1|r_0^1|^2}
{\|Ar_0\|^2}\ .
\]
On the other hand $r_1=r_0-\lambda_0Ar_0$, hence
$r_1^1=(1-\lambda_0\mu_1)r_0^1$, and
$r_1^2=(1-\lambda_0\mu_2)r_0^2$.
Therefore
\begin{equation}
\label{eq:r0r1}
\frac{r_1^2}{r_1^1}=
\frac{1-\lambda_0\mu_2}{1-\lambda_0\mu_1}
\cdot\frac{r_0^2}{r_0^1}
=\frac{-\bar\mu_1|r_0^1|^2}{\bar\mu_2|r_0^2|^2}\frac{r_0^2}{r_0^1}
=\frac{-\bar\mu_1}{\bar\mu_2}\frac{\bar r_0^1}{\bar r_0^2}
\Rightarrow
\frac{|r_1^2|}{|r_1^1|}=
\frac{|\mu_1|}{|\mu_2|}
\frac{|r_0^1|}{|r_0^2|}
\ .
\end{equation}
By applying the same procedure to $r_1$ instead of $r_0$, we obtain
\[
\frac{|r_2^2|}{|r_2^1|}=\frac{|r_0^2|}{|r_0^1|}
\Rightarrow \frac{|r_2^1|^2+\zeta|r_2^2|^2}{|r_2^1|^2+|r_2^2|^2}
=\frac{|r_0^1|^2+\zeta|r_0^2|^2}{|r_0^1|^2+|r_0^2|^2}
,\quad \text{i.e. } \omega_2=\omega_0\ .
\]
This shows that the sequence $\omega_n$ is periodic with period $2$.
With the above formulae for $1-\lambda\mu_1$ and $1-\lambda\mu_2$,
we also have
\begin{equation}
\label{quotient}
\frac{\|r_1\|^2}{\|r_0\|^2}=
\frac{|1-\lambda\mu_1|^2|r_0^1|^2+|1-\lambda_1\mu_2|^2|r_0^2|^2}
{|r_0^1|^2+|r_0^2|^2}
=\rho^2|1-\zeta|^2\frac{|r_0^1|^2|r_0^2|^2}{\|r_0\|^2\|Ar_0\|^2}\ .
\end{equation}
Let $y=|r_0^2|^2/|r_0^1|^2$. The above fraction equals, up to a constant,
\[\frac{1}{1+1/y}\cdot
\frac{1}{1+y|\mu_2|^2/|\mu_1|^2}=:g(y)\ .\]
Because of (\ref{eq:r0r1}),
substituting $r_1$ for $r_0$ amounts to substituting $y$ by
$\frac{|\mu_1|^2}{|\mu_2|^2}\frac{1}{y}$. This does not change the
value of $g(y)$, which means that
$\frac{\|r_2\|^2}{\|r_1\|^2}=\frac{\|r_1\|^2}{\|r_0\|^2}$.
This proves that~$q_2=q_1$. Similarly, $q_n=q_{n-1}$  for $n\ge 2$.
\end{proof}


\subsection{Convergence of $q_n$ to $\rho$}
\label{ssec:convfindim}
We have already seen that {\color{black}$\lim_{n\to \infty} q_n=\rho$} if and only if
{\color{black}$\lim_{n\to \infty}\lambda_n=1$}. 
In this section we will work with the quantities
$$\beta_n:=1-\lambda_n,\  \ u_n:=\omega_{n,1},\ \ \mathrm{and}\  v_n:=\omega_{n,2},$$ 
and we formulate sufficient conditions that 
guarantee $\beta_n\to 0$.  We have
\be
\label{eq:resrec}
r_{n+1}=r_n-\lambda_n(I+\rho U)r_n=\beta_nAr_n-\rho Ur_n\ ,
\ee
and
\be
(r_{n+1}, r_{n+1})=(r_{n+1}, r_n-\lambda_n Ar_n)=(r_{n+1}, r_n)\ .
\ee
Further, since $U$ is unitary,
\be
\label{eq:Aineq}
(1-\rho)\le \nnorm{A}\le (1+\rho)\ .
\ee
Now,
\bes
1-\lambda_{n+1} 
&=&
1-\frac{(r_{n+1}, Ar_{n+1})}{\|Ar_{n+1}\|^2}
=\rho\frac{(Ur_{n+1}, Ar_{n+1})}{\|Ar_{n+1}\|^2}
\\
&=&
\rho\frac{(Ur_{n+1}, r_{n+1})+\rho(r_{n+1}, r_{n+1})}
{\|Ar_{n+1}\|^2}
=\rho\frac{(Ur_{n+1}, r_{n+1})+\rho(r_{n+1}, r_n)}
{\|Ar_{n+1}\|^2}
\\
&=&\rho\frac{(Ur_{n+1}, r_{n+1})+(Ur_{n+1}, \rho Ur_n)}
{\|Ar_{n+1}\|^2}
=\rho\frac{(Ur_{n+1}, r_{n+1}+\rho Ur_n)}
{\|Ar_{n+1}\|^2}
\\
&\stackrel{\eqref{eq:resrec}}{=}&
\rho\frac{(Ur_{n+1}, (1-\lambda_n)Ar_n)}{\|Ar_{n+1}\|^2}
=\rho(1-\bar\lambda_n)
\frac{(Ur_{n+1}, Ar_n)}
{\|Ar_{n+1}\|^2} \;.
\ees
Therefore
\bes
\beta_{n+1}&=&
\rho\bar\beta_n
\frac{(Ur_{n+1}, Ar_n)}{\|Ar_{n+1}\|^2}
=\rho \bar\beta_n
\frac{(U(\beta_n(I+\rho U)r_n-\rho Ur_n), (I+\rho U)r_n)}
{\|Ar_{n+1}\|^2}
\\
&=& 
\rho\bar\beta_n
\frac{\beta_n((U+\rho U^2)r_n, (I+\rho U)r_n)
-(\rho U^2r_n, (I+\rho U)r_n)}
{\|Ar_{n+1}\|^2}
\\
&=&
\rho\bar\beta_n(\beta_n((1+\rho^2)u_n+\rho(1+v_n))-\rho^2u_n-\rho v_n)
\frac{\|r_n\|^2}{\|Ar_{n+1}\|^2}\;.
\ees
Next, the statement
\bes
\|r_{n+1}\|=\|\beta_nAr_n-\rho Ur_n\|
\geq \rho\|Ur_n\|-\|\beta_nAr_n\|
\geq\|r_n\|(\rho-|\beta_n|(1+\rho))
\ees
implies
\bes
\frac{\|r_n\|}{\|Ar_{n+1}\|}
=\frac{\|r_n\|}{\|r_{n+1}\|}
\frac{\|r_{n+1}\|}{\|Ar_{n+1}\|}
\leq
\frac{1}{(\rho-\abs{\beta_n}(1+\rho))(1-\rho)} \ .
\ees
Therefore
\begin{equation}
\label{eq:bineq}
|\beta_{n+1}|
\leq 
\rho|\beta_n|
\frac{(|\beta_n|((1+\rho^2)|u_n|+\rho(1+|v_n|))+\rho^2|u_n|+\rho|v_n|}
{(\rho-|\beta_n|(1+\rho))^2(1-\rho)^2}\ .
\end{equation}
Next we need to estimate $\abs{u_n}, \abs{v_n}$. We have 
\bes
u_{n+1}
&=&
\frac{(Ur_{n+1}, r_{n+1})}{(r_{n+1}, r_{n+1})}
=\frac{(U(\beta_nAr_n-\rho Ur_n), \beta_nAr_n-\rho Ur_n)}
{\|r_{n+1}\|^2}
\\
&=&
\frac{|\beta_n|^2(AUr_n, Ar_n)-\rho(\beta_n(Ar_n,r_n)
+\bar\beta_n(U^2r_n, Ar_n))
+\rho^2(Ur_n,r_n)}
{\|r_{n+1}\|^2}\;,
\ees
hence
\be
|u_{n+1}|
&=& \left(|\beta_n|^2\|A\|^2+2\rho|\beta_n|\cdot\|A\|+\rho^2|u_n|\right)
\frac{\|r_n\|^2}{\|r_{n+1}\|^2}
\\
&\leq&
\frac{|\beta_n|^2(1+\rho)^2+2\rho|\beta_n|(1+\rho)+\rho^2|u_n|}
{(\rho-|\beta_n|(1+\rho))^2}\,.
\ee
The analogous inequality can be derived for $v_n$. 
We summarize the previous inequalities in 
\begin{proposition}
\label{andrei_inequality}
The following recurrence relations hold:
\begin{equation}
\begin{split}
|\beta_{n+1}|&\leq
\rho|\beta_n|\cdot
\frac{|\beta_n|[(1+\rho^2)|u_n|+\rho(1+|v_n|)]+\rho^2|u_n|+\rho|v_n|}
{[\rho-|\beta_n|(1+\rho)]^2(1-\rho)^2}\ ,
\\
|u_{n+1}|&\leq
\frac{|\beta_n|^2(1+\rho)^2+2\rho|\beta_n|(1+\rho)+\rho^2|u_n|}
{(\rho-|\beta_n|(1+\rho))^2}\ ,
\\
|v_{n+1}|&\leq 
\frac{|\beta_n|^2(1+\rho)^2+2\rho|\beta_n|(1+\rho)+\rho^2|v_n|}
{(\rho-|\beta_n|(1+\rho))^2}\ .
\end{split}
\end{equation}
\end{proposition}
We will also need the following inequality which we state without proof. 
\begin{lemma}
For $|x|\leq 0.1$, $\frac{1}{(1-x)^2}\leq 1+Cx$, with $C=2.5$.
\end{lemma}
\begin{proposition}
\label{prop:ortho1technical}
Assume the following:
$0<\rho<0.1$, and 
$\omega_{0,1}=\omega_{0,2}=\omega_{0,3}=0$.
Then, for $n\geq 1$, we have:
\begin{enumerate}\vspace{7pt}
\item[\textnormal{(}i\textnormal{)}] $|u_n|\leq \rho+2.7\sum_{k=2}^n\rho^k\leq \rho+3\rho^2$\ \textnormal{;}\ \vspace{7pt}
\item[\textnormal{(}ii\textnormal{)}] $|v_n|\leq2.7\sum_{k=2}^n\rho^k\leq 3\rho^2$\ \textnormal{;} \vspace{7pt}
\item[\textnormal{(}iii\textnormal{)}] $|\beta_n|\leq\rho^{n+2}$.\vspace{7pt}
\end{enumerate}
\end{proposition}
\begin{proof}
We use the recurrence relations (\ref{eq:omegarec}) to 
compute the first few terms in the sequences $\beta_n, u_n, v_n$. 
\[\begin{split}
&\tau_0=\frac{\omega_{0,1}+\rho}{\rho \bar\omega_{0,1}+1}=\rho\ ,\quad
\beta_0=\frac{\rho \tau_0}{1+\rho \tau_0}=\frac{\rho^2}{1+\rho^2}\ ,
\\
&u_1=\omega_{1,1}=
\frac{(1+|\tau_0|^2) \omega_{0,1}-\tau_0\bar \omega_{0,0}-\bar \tau_0\omega_{0,2}}
{1+|\tau_0|^2-2\Re(\bar \tau_0\omega_{0,1})}=\frac{-\rho}{1+\rho^2}\ ,
\\
&
v_1=\omega_{1,2}=
\frac{(1+|\tau_0|^2) \omega_{0,2}-\tau_0\bar \omega_{0,1}-\bar \tau_0\omega_{0,3}}
{1+|\tau_0|^2-2\Re(\bar \tau_0\omega_{0,1})}=0\ ,
\\
&\tau_1=\frac{\omega_{1,1}+1}{\rho\bar\omega_{1,1}+1}
=\rho^3,\quad
\beta_1=\frac{\rho \tau_1}{1+\rho \tau_1}=\frac{\rho^4}{1+\rho^4}\ .
\end{split}\]
The inequalities in the proposition are thus true for $n=1$, and
we proceed by induction.  We assume that the statements (i-iii) are true for
some $n\geq 1$, and we prove that they hold for $n+1$ as well.
For that, we rely on the inequalities of 
Proposition \ref{andrei_inequality}. 
We start with the inequality (iii):
\[\begin{split}
|\beta_{n+1}|
&\leq
\rho^{n+3}\times
\frac{\rho^{n+2}[(1+\rho^2)(\rho+3\rho^2)+\rho(1+3\rho^2)]
+\rho^2(\rho+3\rho^2)+3\rho^3}
{[\rho-\rho^{n+2}(1+\rho)]^2(1-\rho)^2}
\\
&=
\rho^{n+3}\times
\frac{\rho^{n+1}[(1+\rho^2)(1+3\rho)+1+3\rho^2]
+\rho(1+3\rho)+3\rho}
{[1-\rho^{n+1}(1+\rho)]^2(1-\rho)^2}
\\
&\leq
\rho^{n+3}\times
\frac{\rho^2[(1+\rho^2)(1+3\rho)+1+3\rho^2]
+\rho(1+3\rho)+3\rho}
{[1-\rho^2(1+\rho)]^2(1-\rho)^2}\ .
\end{split}\]
The fraction on the right hand side has numerator equal to
$4\rho+5\rho^2+3\rho^3+4\rho^4+3\rho^5$. 
This is easily seen to be less than $0.5$, as $0<\rho<0.1$. 
On the other hand, the denominator is certainly greater than
$0.9^2\times(1-\frac{1.1}{100})^2>0.7$. Therefore the fraction on 
right hand side is less than $1$, and $|\beta_{n+1}|\leq\rho^{n+3}$.
\\\indent
For inequality (ii),
\[\begin{split}
|u_{n+1}|&\leq
\frac{\rho^{2(n+2)}(1+\rho)^2+2\rho^{n+3}(1+\rho)+\rho^2|u_n|}
{[\rho-\rho^{n+2}(1+\rho)]^2}
\\
&=
\frac{\rho^{2(n+1)}(1+\rho)^2+2\rho^{n+1}(1+\rho)+|u_n|}
{[1-\rho^{n+1}(1+\rho)]^2}
\\
&=\frac{x^2+2x+|u_n|}{(1-x)^2}\ ,\quad \text{with}\; x=\rho^{n+1}(1+\rho),
\\
&\leq
(1+Cx)(x^2+2x+|u_n|)\ ,\quad\text{with}\; C=2.5,
\\
&=
|u_n|+x[2+C|u_n|+(2C+1)x+Cx^2]\ .
\end{split}\]
From the induction step, $|u_n|\leq \rho+3\rho^2$. 
Also, $x=\rho^{n+1}(1+\rho)\leq\rho^2(1+\rho)$. 
The quantity inside the square brackets is less
than 
$$2+C(\rho+3\rho^2)+(2C+1)\rho^2(1+\rho)+C\rho^4(1+\rho)^2\ .$$ 
As $0<\rho<1$, this is easily seen to be less than $2.5$. Therefore,
$$|u_{n+1}|\leq |u_n|+2.5(1+\rho)\rho^{n+1}<|u_n|+2.7\rho^{n+1}\ .$$ 
Hence $|u_{n+1}|\leq |u_1|+2.7\sum_{k=2}^{n+1}\rho^k$. 
The exact same method is applied to $v_{n+1}$.
\end{proof}

\begin{theorem}
\label{th:ortho1proof}
Assume the following hold:
\begin{enumerate}
\item[\textnormal{(}a\textnormal{)}] $0<\rho<0.1$\textnormal{;}
\item[\textnormal{(}b\textnormal{)}] $d\geq 4$\textnormal{;}
\item[\textnormal{(}c\textnormal{)}] $r_0=[1,\dots, 1]^T$\textnormal{;}
\item[\textnormal{(}d\textnormal{)}] $\zeta_k$ are the roots of unity of order $d$\textnormal{;}
\item[\textnormal{(}e\textnormal{)}] $A=I+\rho U$, $U=diag([\zeta_1,\dots, \zeta_d])$.
\end{enumerate}
Then the sequence 
$r_{n+1}=r_n-\Pi_{Ar_n}r_n$
satisfies
\begin{equation}
\lim_{n\to\infty}
\frac{\|r_{n+1}\|}{\|r_n\|}=\rho\ .
\end{equation}
\end{theorem}
\begin{proof}
The hypotheses ensure that
$\omega_{0,1}=\omega_{0,2}=\omega_{0,3}=0$.
Proposition~\ref{prop:ortho1technical} then applies to show 
\[
{\color{black}\lim_{n\to\infty}\beta_n= 0\Rightarrow\lim_{n\to\infty}\lambda_n=1
\stackrel{(\ref{eq:convequiv})}{\Rightarrow} \lim_{n\to\infty}q_n=\rho\ .
}
\]
\end{proof}
{\color{black}
Note that Theorem~\ref{th:ortho1proof} is a step towards proving Conjecture~\ref{conj:hull} 
for the case when $\zeta_1,\dots,\zeta_d$ are the roots of unity. However, we
should point out that if $d$ is even, then  $-\rho\in \mathrm{Hull}(\zeta_1,\dots,\zeta_d)$ holds for 
all $0 < \rho < 1$; for odd $d$, then $-\rho\in \mathrm{Hull}(\zeta_1,\dots,\zeta_d)$ 
for  $0<\rho\le \cos(\pi/d)$. In Theorem~\ref{th:ortho1proof} 
we show the desired convergence holds for the more restrictive condition $0<\rho<0.1$; thus,
the more general case stated in the Conjecture~\ref{conj:hull} remains an open question, 
even for this example.
}

%% file: conclusions.tex
\section*{Conclusions}
For $k\in \N$ we give examples of linear systems 
for which we conjectured that Orthomin(1),
$\dots$, Orthomin($k$) achieve the same asymptotic convergence
rate. These examples show that, in general, Orthomin($k$) does not converge
faster than Orthomin(1). We analyze in detail the convergence of
Orthomin(1) and provide numerical evidence in support of our
conjectures with respect to Orthomin($k$) for $k>1$.  The analysis for
Orthomin(1) is fairly complicated and we do not see a straightforward
way to extend the arguments to  Orthomin($k$) for $k>1$. 
We provide numerical evidence that certain normal operators
(related to numerical PDEs) with spectrum lying on an ellipse,
have the following property: Orthomin($2$), Orthomin($3$), etc.
all have the same asymptotic convergence rate (depending
only on the ellipse); moreover this is smaller than the
asymptotic convergence rate of Orthomin($1$). This example offers a  promising
path to finding improved convergence rate estimates for Orthomin(2)
under additional assumptions on the spectrum/field of values of the matrix.
An important question, which remains unanswered, is whether there are 
applications where Orthomin($k$), perhaps coupled with preconditioners,
can compete with the usual iterative solvers for non-symmetric systems.

%% file: appendixnumerics.tex
\appendix
\section{\color{black}Numerical evidence supporting Conjecture~\ref{conj:o1tokexamples}}
\label{sec:numresults}
{\color{black}In order to verify numerically the validity of Conjecture~\ref{conj:o1tokexamples},
we conducted several experiments with Orthomin($k$) for the system~\eqref{eq:sysmainex} with $U$ as 
in~\eqref{eq:Udef}; the right hand side is $\lbrack 1, 1\dots, 1\rbrack^T$, and the initial guess is zero, but we
also conducted experiments with random right hand side and initial guess, and the outcomes were very similar. The tolerance was set at $10^{-8}$.
We report the results of computations for $k\in \{1, 2, 3, 7, 11, 13\}$,
$\rho\in \{0.2, 0.5, 0.8\}$, and $d=16$ in Figure~\ref{fig:conv_Omin_d16}, $d=32$ in Figure~\ref{fig:conv_Omin_d32},
and $d=64$ in Figure~\ref{fig:conv_Omin_d64}, respectively, for a total 54 cases. For each case we record the residual norms
and we compare the ratios $q_n=\|r_{n+1}^{(k)}\|/\|r_{n}^{(k)}\|$
of successive residual norms. In Figures~\ref{fig:conv_Omin_d16}--\ref{fig:conv_Omin_d64}
we plot the quantities
$
\log_{10}|q_n-\rho|
$
for each case.

We should point out that for Orthomin($k$) with $k=1, 2, 3$ and all the values of $\rho$ and $d$ that we considered, we have a rapid convergence 
of $q_n$ to $\rho$. However, for $d=16$ and $k\in\{7, 11, 13\}$, 
we notice in Figure~\ref{fig:conv_Omin_d16} that $\lim_{n\to\infty}q_n = \rho$ for the smaller value $\rho=0.2$, but this convergence does not appear to hold for
$\rho \in \{0.5, 0.8\}$; instead, while still relatively small ($\le 10^{-2}$), the absolute difference $|q_n-\rho|$ exhibits an oscillatory behavior.
However, this divergent behavior appears to gradually change towards convergence as we increase $d$, 
as shown in Figures~\ref{fig:conv_Omin_d32}--\ref{fig:conv_Omin_d64}. For $d=32$ we notice 
that $\lim_{n\to\infty}q_n = \rho$ for all the cases, even though it appears to be slightly slower for Orthomin(13); for $d=64$ 
(Figure~\ref{fig:conv_Omin_d64}) w have convergence of $q_n$ to $\rho$ for all the value of $k$ and $\rho$. This indicates that, for a 
fixed $k$, if $d$ is large enough, or $\rho$ is sufficiently small, then  $\lim_{n\to\infty}q_n = \rho$, which is consistent with Conjecture~\ref{conj:o1tokexamples}.

}

\begin{figure}[!htb]
\begin{center}
        \includegraphics[width=\textwidth]{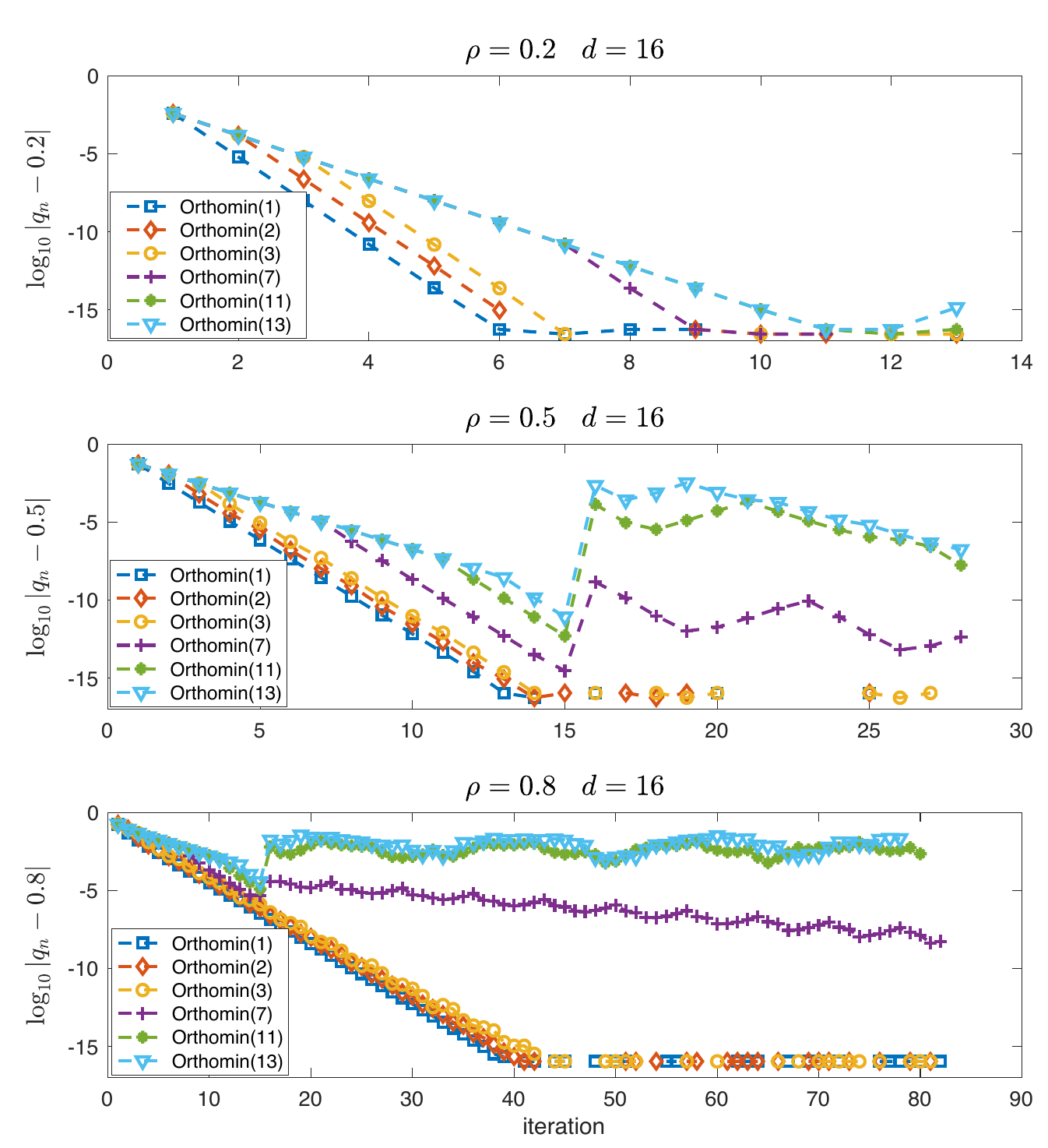}
\caption{Convergence results for $q_n$ to $\rho$  for Orthomin$(k)$, 
  $k\in \{1, 2, 3, 7, 11, 13\}$,   $\rho\in \{0.2, 0.5, 0.8\}$, and $d=16$.} 
\label{fig:conv_Omin_d16}
\end{center}
\end{figure}

\begin{figure}[!htb]
\begin{center}
        \includegraphics[width=\textwidth]{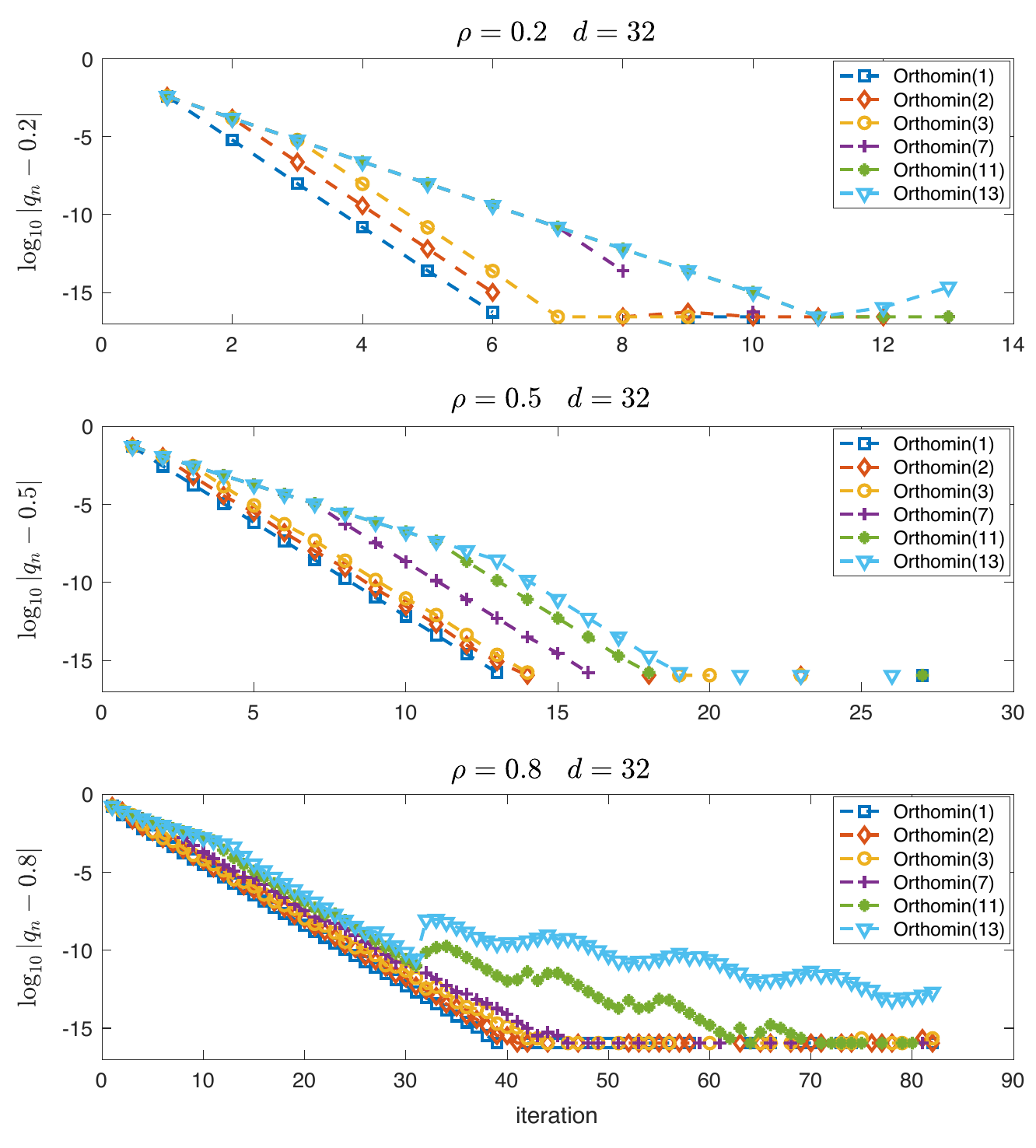}
\caption{Convergence results for $q_n$ to $\rho$  for Orthomin$(k)$, 
  $k\in \{1, 2, 3, 7, 11, 13\}$,   $\rho\in \{0.2, 0.5, 0.8\}$, and $d=32$.} 
\label{fig:conv_Omin_d32}
\end{center}
\end{figure}

\begin{figure}[!htb]
\begin{center}
        \includegraphics[width=\textwidth]{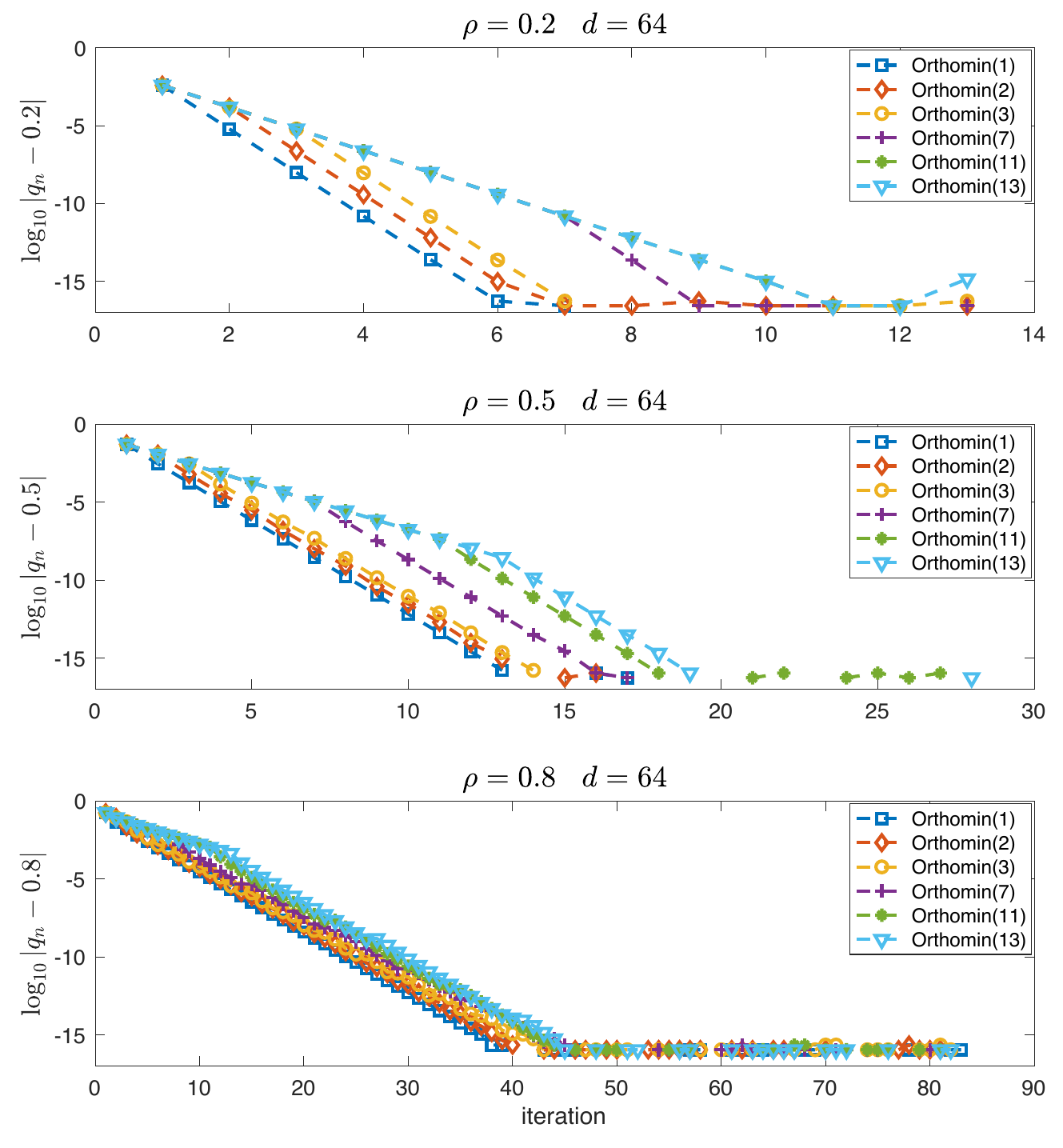}
\caption{Convergence results for $q_n$ to $\rho$  for Orthomin$(k)$, 
  $k\in \{1, 2, 3, 7, 11, 13\}$,   $\rho\in \{0.2, 0.5, 0.8\}$, and $d=64$.} 
\label{fig:conv_Omin_d64}
\end{center}
\end{figure}

%% file: main_jcam.bbl
\begin{thebibliography}{10}
\expandafter\ifx\csname url\endcsname\relax
  \def\url#1{\texttt{#1}}\fi
\expandafter\ifx\csname urlprefix\endcsname\relax\def\urlprefix{URL }\fi
\expandafter\ifx\csname href\endcsname\relax
  \def\href#1#2{#2} \def\path#1{#1}\fi

\bibitem{Vinsome76}
P.~Vinsome, Orthomin, an iterative method for solving sparse sets of
  simultaneous linear equations, in: Proc. of the 4th Symposium on Reservoir
  Simulation, Society of Petroleum Engineers of AIME, 1976, pp. 149--159.

\bibitem{chen2001nonlinear}
Y.~Chen, D.~Cai, Nonlinear {O}rthomin($k$) methods, Applied mathematics and
  computation 124~(3) (2001) 351--363.

\bibitem{MR1804663}
S.-L. Zhang, Y.~Oyanagi, M.~Sugihara, Necessary and sufficient conditions for
  the convergence of {${\rm Orthomin}(k)$} on singular and inconsistent linear
  systems, Numer. Math. 87~(2) (2000).

\bibitem{abe2004variant}
K.~Abe, S.-L. Zhang, T.~Mitsui, C.-H. Jin, A variant of the {O}rthomin($2$)
  method for singular linear systems, Numerical Algorithms 36 (2004) 189--202.

\bibitem{abe2008variant}
K.~Abe, S.-L. Zhang, A variant algorithm of the {O}rthomin($m$) method for
  solving linear systems, Applied mathematics and computation 206~(1) (2008)
  42--49.

\bibitem{li2005comparison}
W.~Li, Z.~Chen, R.~E. Ewing, G.~Huan, B.~Li, Comparison of the {GMRES} and
  {Orthomin} for the black oil model in porous media, International Journal for
  Numerical Methods in Fluids 48~(5) (2005) 501--519.

\bibitem{modak2006new}
R.~Modak, A.~Gupta, New applications of {O}rthomin(1) algorithm for
  k-eigenvalue problem in reactor physics, Annals of Nuclear Energy 33~(6)
  (2006) 538--543.

\bibitem{houzeaux2011extension}
G.~Houzeaux, R.~Aubry, M.~V{\'a}zquez, Extension of fractional step techniques
  for incompressible flows: The preconditioned {O}rthomin(1) for the pressure
  {S}chur complement, Computers \& Fluids 44~(1) (2011) 297--313.

\bibitem{MR1990645}
Y.~Saad, \href{https://doi.org/10.1137/1.9780898718003}{Iterative methods for
  sparse linear systems}, 2nd Edition, Society for Industrial and Applied
  Mathematics, Philadelphia, PA, 2003.
\newblock \href {https://doi.org/10.1137/1.9780898718003}
  {\path{doi:10.1137/1.9780898718003}}.
\newline\urlprefix\url{https://doi.org/10.1137/1.9780898718003}

\bibitem{MR1474725}
A.~Greenbaum, Iterative methods for solving linear systems, Vol.~17 of
  Frontiers in Applied Mathematics, Society for Industrial and Applied
  Mathematics (SIAM), Philadelphia, PA, 1997.

\bibitem{elmanthesis}
H.~C. Elman, Iterative methods for large, sparse, nonsymmetric systems of
  linear equations, Ph.D. thesis, Dept. Computer Science, Yale Univ., New
  Haven, CT, also available as Technical Report 229 (1982).

\bibitem{MR694523}
S.~C. Eisenstat, H.~C. Elman, M.~H. Schultz, Variational iterative methods for
  nonsymmetric systems of linear equations, SIAM J. Numer. Anal. 20~(2) (1983)
  345--357.

\bibitem{MR98b:47008}
K.~E. Gustafson, D.~K.~M. Rao, Numerical range, Universitext, Springer-Verlag,
  New York, 1997.

\bibitem{MR0010938}
J.~E. Wilkins, Jr., A note on skewness and kurtosis, Ann. Math. Statistics 15
  (1944) 333--335.

\bibitem{MR1791507}
C.~A.~J. Klaassen, P.~J. Mokveld, B.~van Es,
  \href{http://dx.doi.org/10.1016/S0167-7152(00)00090-0}{Squared skewness minus
  kurtosis bounded by {$186/125$} for unimodal distributions}, Statist. Probab.
  Lett. 50~(2) (2000) 131--135.
\newblock \href {https://doi.org/10.1016/S0167-7152(00)00090-0}
  {\path{doi:10.1016/S0167-7152(00)00090-0}}.
\newline\urlprefix\url{http://dx.doi.org/10.1016/S0167-7152(00)00090-0}

\bibitem{MR2599773}
E.~Vecharynski, J.~Langou, \href{https://doi.org/10.1137/080727403}{The
  cycle-convergence of restarted {GMRES} for normal matrices is sublinear},
  SIAM J. Sci. Comput. 32~(1) (2010) 186--196.
\newblock \href {https://doi.org/10.1137/080727403}
  {\path{doi:10.1137/080727403}}.
\newline\urlprefix\url{https://doi.org/10.1137/080727403}

\end{thebibliography}
